\documentclass[a4paper]{amsart}
\usepackage{amssymb, amsmath,amsthm,amscd}
\usepackage{hyperref}
\usepackage{ifthen}

\newtheorem{thm}{Theorem}[section]
\newtheorem{prop}[thm]{Proposition}
\newtheorem{lem}[thm]{Lemma}
\newtheorem{cor}[thm]{Corollary}

\theoremstyle{remark}
\newtheorem*{remark}{Remark}

\DeclareMathOperator{\Br}{Br}
\DeclareMathOperator{\Pic}{Pic}

\DeclareMathOperator{\Gal}{Gal}

\newcommand{\kbar}{\bar{k}}
\newcommand{\Xbar}{\bar{X}}
\newcommand{\Ubar}{\bar{U}}
\renewcommand{\H}{\mathrm{H}}
\newcommand{\Q}{\mathbb{Q}}
\renewcommand{\P}{\mathbb{P}}
\newcommand{\Gm}{\mathbf{G}_m}

\newcommand{\Z}{\mathbb{Z}}

\begin{document}

\title{A uniform bound on the Brauer groups of certain log K3 surfaces}

\author{Martin Bright \and Julian Lyczak}

\begin{abstract}
Let $U$ be the complement of a smooth anticanonical divisor in a del Pezzo surface of degree at most 7 over a number field $k$.  We show that there is an effective uniform bound for the size of the Brauer group of $U$ in terms of the degree of $k$.
\end{abstract}

\maketitle
\section{Introduction}

The Brauer group of an algebraic variety is an invariant whose study goes back to Grothendieck~\cite{G-I, G-II, G-III}.  It has both geometric and arithmetic applications, in particular to the study of rational points.  However, the Brauer group of a variety over a non-algebraically closed field is often difficult to compute; even the question of when it is finite is not well understood.  Our main result in this article is the following.

\begin{thm}\label{thm}
Let $k$ be a number field, let $X$ be a smooth del Pezzo surface of degree at most $7$ over $k$, and let $U \subset X$ be the complement of a smooth, irreducible curve $C \in \lvert -K_X \rvert$.   Denote by $m$ the degree $[k:\Q]$, and let $N(m)$ be a bound for the maximum order of a torsion point on an elliptic curve over a number field having degree $m$ over $\Q$.  Then the order of $\Br U / \Br k$ is bounded by
\[
\# (\Br U/\Br k) < 2^{14} N(240m)^2.
\]
\end{thm}

That such a bound $N(m)$ exists was first proved by Merel~\cite{Merel}.  Specifically, Merel proved that, for an elliptic curve $E$ over a number field $k$ of degree $m$, any prime $p$ dividing the order of a torsion element of $E(k)$ is bounded by $m^{3m^2}$.  Together with previously-known results of Faltings and Frey, this proves the existence of a bound $N(m)$ on the maximal order of a torsion element in $E(k)$, but does not give an effective determination of $N(m)$.  Parent~\cite{Parent} remedied the situation by proving that, if $p>3$ is a prime such that $E(k)$ admits a torsion point of order $p^n$, then the inequality $p^n \le 65(3^m-1)(2m)^6$ holds.  (Parent also gave similar bounds for $p=2,3$.)  Together, these results give an effective bound $N(m)$ on the order of a torsion point on an elliptic curve over $k$.

We begin by putting our result in context.
Let $k$ denote a number field, and let $X$ be a smooth, proper, geometrically integral variety over $k$.  Let $\kbar$ be an algebraic closure of $k$, and denote by $\Xbar$ the base change of $X$ to $\kbar$.  The Brauer group $\Br X$ is often split into two parts: the \emph{algebraic Brauer group} is $\Br_1 X = \ker(\Br X \to \Br \Xbar)$, and the \emph{transcendental Brauer group} is the quotient $\Br X / \Br_1 X$.  Of these, the algebraic Brauer group is easier to understand, in part owing to an isomorphism $\Br_1 X/ \Br k \cong \H^1(k, \Pic\Xbar)$ coming from the Hochschild--Serre spectral sequence.  The first interesting case is when $X$ is a geometrically rational surface; here we have $\Br\Xbar = 0$, and so the Brauer group can be calculated once the Galois action on the finitely-generated group $\Pic\Xbar$ is known.  In the case of del Pezzo surfaces, all possibilities for the finite group $\Br X / \Br k$ have been tabulated: see \cite[Theorem~1.4.1]{Corn}.

A more complicated case is that of K3 surfaces.  Here $\Br \Xbar$ is infinite, but it was proved by Skorobogatov and Zarhin \cite{SZ} that the quotient $\Br X / \Br_1 X$ is finite.  The question then arises of trying to bound this finite group; there has been quite a body of work on this in recent years.  Ieronymou, Skorobogatov and Zarhin proved in~\cite{ISZ} that, when $X$ is a diagonal quartic surface over the field $\Q$ of rational numbers, the order of $\Br X / \Br \Q$ divides $2^{25}\times 3^2\times 5^2$.  When $X$ is the Kummer surface associated to $E\times E$, with $E/\Q$ an elliptic curve with complex multiplication, Newton~\cite{Newton} described the odd-order part of $\Br X/\Br \Q$.  When $X$ is the Kummer surface associated to a curve of genus $2$ over a number field $k$, Cantoral Farf\'an, Tang, Tanimoto and Visse~\cite{CTTV} described an algorithm for computing a bound for $\Br X / \Br k$.  More generally, V\'arilly-Alvarado~\cite[Conjectures~4.5, 4.6]{VA} has conjectured that there should be a uniform bound on $\Br X / \Br k$ for any K3 surface $X$, depending only on the geometric Picard lattice of the surface.
Recent progress towards this conjecture has been made by V\'arilly-Alvarado and Viray for certain Kummer surfaces associated to non-CM elliptic curves~\cite[Theorem~1.8]{VAV} and by Orr and Skorobogatov for K3 surfaces of CM type~\cite[Corollary~C.1]{OS}.

So far we have been discussing proper varieties.  However, non-proper varieties are also of arithmetic interest.  A particular case is that of log K3 surfaces; the arithmetic of integral points on log K3 surfaces shows several features analogous to those of rational points on proper K3 surfaces.  See~\cite{Harpaz} for an introduction to the arithmetic of log K3 surfaces.  One example of a log K3 surface is the complement of an anticanonical divisor in a del Pezzo surface, and it is that case with which we concern ourselves in this note.

Some calculations of the Brauer groups of such varieties have already appeared in the literature.  In~\cite{CTW}, Colliot-Th\'el\`ene and Wittenberg computed explicitly the Brauer group of the complement of a plane section in certain cubic surfaces.  In~\cite{JS}, Jahnel and Schindler carried out extensive calculations in the case of a del Pezzo surface of degree $4$.
In this note, we compute the possible algebraic Brauer groups of these surfaces, and use uniform boundedness of torsion of elliptic curves to bound the possible transcendental Brauer groups, resulting in Theorem~\ref{thm}.

\begin{remark}
One might naturally ask about the cases of degree $8$ or $9$.
The proof of Theorem~\ref{thm} depends on the fact that del Pezzo surfaces of degree $d\leq 7$ contain exceptional curves. Our proof also shows that the result holds for del Pezzo surfaces of degree $8$ which are geometrically isomorphic to $\mathbb P^2_{\kbar}$ blown up in a point. On the other hand, for the remaining del Pezzo surfaces of degree $8$ and $9$, i.e.~geometrically isomorphic to $\mathbb P^1_{\kbar}\times \mathbb P^1_{\kbar}$ and $\mathbb P^2_{\kbar}$ respectively, the result does not hold: the torsion in $\Pic \Ubar$ in those cases shows that $\Br_1 U/ \Br k$ is infinite.  
Even if one asks only about the transcendental part of $\Br U$, our proof does not show finiteness; see~\cite[Proposition~5.3]{CTW} for a description of the case $d=9$, that is, the complement of a smooth cubic plane curve.
\end{remark}

\section{Proof of the theorem}

Throughout this section, let $X$ be a del Pezzo surface of degree $d$ over a number field $k$ and let $C \in \lvert -K_X \rvert$ be an anticanonical curve on $X$; we assume $C$ to be smooth and irreducible. The adjunction formula shows that $C$ has genus $1$.  We will consider the quasi-projective variety $U = X \setminus C$.   Let $\kbar$ be an algebraic closure of $k$ and let $\Xbar$ and $\Ubar$ denote the base changes to $\kbar$ of $X$ and $U$ respectively.  By a \emph{line} on $X$ or on $\Xbar$ we mean an irreducible curve $L$ satisfying $L \cdot L = L \cdot K_X = -1$.

Recall that the \emph{algebraic Brauer group} of $U$ is defined to be $\Br_1 U = \ker(\Br U \to \Br \Ubar)$.  The quotient $\Br U / \Br_1 U$, isomorphic to the image of $\Br U \to \Br \Ubar$, will be called the \emph{transcendental Brauer group} of $U$.  We will find independent bounds for the algebraic and the transcendental Brauer groups of such log K3 surfaces.

\subsection{The algebraic Brauer group}

Suppose that $X$ has degree $d \le 7$.   Then $\Xbar$ is isomorphic to the blow-up of $\P^2$ in $9-d$ points, and so the Picard group of $\Xbar$, together with its intersection pairing, depends only on $d$.  The Galois group $\Gal(\kbar/k)$ acts on $\Pic \Xbar$ preserving intersection numbers, and so the action factors through the isometry group of the lattice, which is known to be the Weyl group of a particular root system.  See~\cite[Sections~25--26]{Manin} for details.

We have $\Br\Xbar=0$, since $\Xbar$ is a rational surface (see \cite[Theorem~42.8]{Manin}).  Therefore $\Br X$ can be computed using the Hochschild--Serre spectral sequence
\[
\H^p(k, \H^q(\Xbar,\Gm)) \Rightarrow \H^{p+q}(X, \Gm).
\]
Using $\H^0(\Xbar,\Gm) = \kbar^\times$, the exact sequence of low-degree terms includes
\[
\Br k \to \Br_1 X \to \H^1(k, \Pic\Xbar) \to \H^3(k,\Gm).
\]
Because $k$ is a number field, we have $\H^3(k,\Gm)=0$ and therefore an isomorphism
\[
\Br_1 X / \Br k \cong \H^1(k, \Pic\Xbar).
\]
(On the left we abuse notation slightly: the map $\Br k \to \Br_1 X$ need not be injective, but we still write its cokernel as $\Br_1 X / \Br k$).

If $K/k$ is the minimal field extension over which all of $\Pic\Xbar$ is defined, then $\H^1(K, \Pic\Xbar)$ is trivial (since $\Pic\Xbar$ is torsion-free) and the inflation-restriction sequence gives an isomorphism $\H^1(K/k, \Pic X_K) \cong \H^1(k,\Pic\Xbar)$.  The Galois group $\Gal(K/k)$ acts faithfully on $\Pic X_K$, so can be identified with a subgroup of the isometry group of the lattice $\Pic X_K$.  Thus the finitely many possibilities for $\H^1(k,\Pic\Xbar)$ may be computed by running through all subgroups $G$ of the appropriate Weyl group and calculating the resulting $\H^1(G,\Pic\Xbar)$.  For $d \ge 5$ the cohomology group is always trivial: see~\cite[Theorem~29.3]{Manin}.  The calculation for $d=3,4$ has been carried out by Swinnerton-Dyer~\cite{SD} and for $d=1,2$ by Corn~\cite[Theorem~1.4.1]{Corn}.

For the open subvariety $U \subset X$, exactly the same approach yields a calculation of the group $\Br_1 U / \Br k$, as described by the following proposition.

\begin{table}[tbp]
\begin{center}
\newcommand{\sep}{\quad}
\begin{tabular}{rll}
Degree & $\Br_1 X/ \Br k$ & Possibilities for $\Br_1 U/ \Br k$\\
\hline
$\mathbf{d = 7}$
& $1$ & $1$ \\
\hline
$\mathbf{d = 6}$
& $1$ & $1$ \sep$2$ \sep $3$ \sep $6$ \\
\hline
$\mathbf{d = 5}$
& $1$ & $1$ \sep $5$ \\
\hline
$\mathbf{d = 4}$
& $1$ & $1$ \sep $2$ \sep $2^2$ \sep $2^3$ \sep $2^4$ \sep $4$ \sep $2\cdot 4$ \\
& $2$ & $2$ \sep $2^2$ \sep $2^3$ \sep $2^4$ \sep $2\cdot 4$ \\
& $2^2$ & $2^3$ \sep $2^4$ \sep $2^2\cdot 4$ \\
\hline
$\mathbf{d = 3}$
& $1$ & $1$ \sep $3$ \sep $3^2$ \\
& $2$ & $2$ \sep $6$ \\
& $2^2$ & $2^2$ \sep $2\cdot 6$ \\
& $3$ & $3$ \sep $3^2$ \\
& $3^2$ & $3^3$ \\
\hline
$\mathbf{d = 2}$
& $1$ & $1$ \sep $2$ \\
& $2$ & $2$ \sep $2^2$ \sep $2\cdot 4$ \sep $4$ \\
& $2^2$ & $2^2$ \sep $2^3$ \sep $2\cdot 4$ \sep $2^2\cdot 4$ \sep $4^2$ \\
& $2^3$ & $2^3$ \sep $2^4$ \sep $2^2\cdot 4$ \\
& $2^4$ & $2^4$ \sep $2^5$ \sep $2^3\cdot 4$ \\
& $2^5$ & $2^6$ \\
& $2^6$ & $2^7$ \\
& $3$ & $3$ \sep $6$ \\
& $3^2$ & $3^2$ \sep $3\cdot 6$ \\
& $2\cdot 4$ & $2\cdot 4$ \sep $2^2\cdot 4$ \sep $4^2$ \\
& $2^2\cdot 4$ & $2^2\cdot 4$ \sep $2^3\cdot 4$ \\
& $4$ & $2\cdot 4$ \sep $4$ \\
& $4^2$ & $2\cdot 4^2$ \\
\hline
\end{tabular}
\end{center}
\caption{Possible group structures of $\Br_1 U / \Br k$.  For example, $2^2 \cdot 4$ means $(\Z/2\Z)^2 \times \Z/4\Z$.}
\label{table}
\end{table}

\begin{prop}\label{prop:possiblebrauers}
Let $U \subset X$ be as in Theorem~\ref{thm}.  Then $\Br_1 U/\Br k$ depends only on $\Pic\Xbar$ as a Galois module and its order is at most $256$.  For $d=1$, the natural map $\Br_1 X \to \Br_1 U$ is an isomorphism. For $2 \le d \le 7$, the possible combinations of $\Br X / \Br k$ and $\Br_1 U / \Br k$ are as shown in Table~\ref{table}.
\end{prop}

Note that our computations agree with those of Jahnel and Schindler~\cite[Remark~4.7i)]{JS} on del Pezzo surfaces of degree $4$.

\begin{proof}
As $C$ is irreducible, a section of $\Gm$ on $\Ubar$ corresponds to a rational function on $\Xbar$ whose divisor is a multiple of $C$. The intersection of this principal divisor with $C$ must be zero and we see that $\H^0(\Ubar,\Gm)=\kbar^\times$.  As above, the Hochschild--Serre spectral sequence gives an isomorphism $\Br_1 U/\Br k \to \H^1(k,\Pic \Ubar)$.
By~\cite[Proposition II.6.5]{Hartshorne} we have an exact sequence of Galois modules
\[
0 \to \Z \to \Pic\Xbar \to \Pic\Ubar \to 0
\]
where the first maps sends $1$ to the anticanonical class in $\Pic \Xbar$. 
Enumerating all possible Galois actions on $\Pic\Xbar$ allows us to calculate the possible cohomology groups.  For \textsc{Magma} code to accomplish this calculation, see~\cite{code}.  In the case $d=1$, the following lemma and the list in Corn~\cite[Theorem~1.4.1]{Corn} spares us from what would be a lengthy calculation.
\end{proof}

\begin{lem}
The natural map $\H^1(k, \Pic \Xbar) \to \H^1(k, \Pic\Ubar)$ is injective and the cokernel has exponent dividing $\delta$, where $\delta$ is the minimal non-zero value of $\lvert D \cdot K_X \rvert$ for $D$ a divisor on $X$.
\end{lem}

\begin{proof}
Let $D$ be a divisor with $D \cdot K_X = \delta$.  As above, we have the exact sequence of Galois modules
\[
0 \to \Z \xrightarrow{i} \Pic\Xbar \to \Pic\Ubar \to 0.
\]
The map $E \mapsto E \cdot D$ gives a map $s \colon \Pic \Xbar \to \mathbb Z$ with the property that $s\circ i$ is multiplication by $\delta$. Consider the following part of the long exact sequence associated to this short exact sequence:
\[
0 \to \H^1(k, \Pic\Xbar) \to \H^1(k, \Pic\Ubar) \to \H^2(k,\Z) 
\mathrel{\mathop{\rightleftarrows}^i_s} \H^2(k, \Pic\Xbar)
\]
We see that the cokernel of $\H^1(k, \Pic\Xbar) \to \H^1(k, \Pic\Ubar)$ is isomorphic to the kernel of $i$, which is contained in the kernel of $s\circ i$; but this map is multiplication by $\delta$.
\end{proof}

\begin{cor}\label{cor:isombrauergrps}
If $X$ is a del Pezzo surface of degree $1$ or $X$ contains a line defined over $k$, then the map $\Br_1 X/\Br k \to \Br_1 U/\Br k$ is an isomorphism.
\end{cor}

\begin{remark}
It follows from the calculations that, of the 19 possible Galois module structures on $\Pic \Xbar$ in the case $d=5$, only one yields a non-trivial algebraic Brauer group. Using the results in \cite{dP5s} one can construct such log K3 surfaces and even write down a cyclic Azumaya algebra generating this group.
\end{remark}

\subsection{The transcendental Brauer group}

First we will prove Theorem~\ref{thm}, following techniques used by Colliot-Th\'el\`ene and Wittenberg~\cite{CTW}, in the situation where at least one of the lines on $\Xbar$ is defined over $k$.  Note that in this case, by Corollary~\ref{cor:isombrauergrps}, the image of $\Br X$ in $\Br U$ coincides with the algebraic Brauer group $\Br_1 U$.

\begin{lem}\label{lem:finitecokeroverext}
Suppose that a line $L \subseteq X$ is defined over $k$.  Let $m$ denote the degree $[k:\Q]$.  Then the restriction map $\Br X \to \Br U$ is injective, and the order of its cokernel is bounded by $N(m)^2$.
\end{lem}

\begin{proof}
Since $C$ is smooth, we have the exact sequence
\[
0 \to \Br X \to \Br U \xrightarrow{\delta_C} \H^1(C,\Q/\Z)
\]
arising from Grothendieck's purity theorem~\cite[Corollaire~6.2]{G-III}.
So it is enough to bound the image of the residue map $\delta_{C}$.

We have $C\cdot L=1$ and so $C$ and $L$ intersect transversely at a unique $k$-point $P$ on $X$.  
This gives the following commutative diagram:
\[
\begin{CD}
\Br X @>>> \Br U @>{\delta_C}>> \H^1(C,\Q/\Z) \\
@VVV @VVV @VV{\alpha}V \\
\Br L @>>> \Br(L \setminus P) @>{\delta_P}>> \H^1(P,\Q/\Z)
\end{CD}.
\]
We have $L \cong \P^1$ and $L \setminus P \cong \mathbb{A}^1$, both of which have Brauer group isomorphic to $\Br k$; exactness of the bottom row shows that $\delta_P$ is the zero map.
This implies that the image of $\delta_{C}$ is contained in the kernel of $\alpha$.
The first cohomology group $H^1(C,\mathbb Q/\mathbb Z)$ classifies cyclic Galois covers of $C$, and $\ker(\alpha)$ corresponds to those cyclic Galois covers $D \to C$ for which the fibre above $P$ is a trivial torsor for the structure group $\Z/n\Z$. So we consider such covers whose kernel is a disjoint union of $k$-points. 

We first bound the degree of such a cover.  Let $\pi \colon D \to C$ be a cyclic cover of degree $n$, and suppose that the fibre $F=\pi^{-1}(P)$ is trivial, so that $F$ consists of $n$ distinct $k$-points. The Riemann--Hurwitz formula shows that $D$ has genus $1$. Pick a point $Q$ in the fibre $F$.
If we regard $D$ and $C$ as elliptic curves with base points $Q$ and $P$ respectively, then $\pi$ is an isogeny of elliptic curves, and $F(k) = \ker(\pi)$ is a cyclic subgroup of order $n$ in $D(k)$.
In particular, since $D$ is an elliptic curve over $k$ with a point of order $n$, we have $n \le N(m)$.

We now fix $n$ to be the maximal order of an element of $\ker(\alpha)$.  The exponent of a finite Abelian group is equal to the maximal order of its elements, so every element of $\ker(\alpha)$ has order dividing $n$.

Looking at the long exact sequence in cohomology associated to the short exact sequence of sheaves
\[
0 \to \Z/n\Z \to \Q/\Z \xrightarrow{\times n} \Q/\Z \to 0
\]
shows that the natural map $\H^1(C,\Z/n\Z) \to \H^1(C,\Q/\Z)$ is injective.
This identifies $\H^1(C,\Z/n\Z)$ with the $n$-torsion in $\H^1(C,\Q/\Z)$, which contains $\ker(\alpha)$.  The Hochschild--Serre spectral sequence gives a short exact sequence
\[
0 \to \H^1(k,\Z/n\Z) \xrightarrow{\beta} \H^1(C,\Z/n\Z) \to \H^1(\bar{C},\Z/n\Z)
\]
in which the map $\alpha$ induces a left inverse to $\beta$.  Thus $\ker(\alpha)$ is identified with a subgroup of $\H^1(\bar{C},\Z/n\Z)$, which is isomorphic to the $n$-torsion in $\Pic \bar{C}$ and so has order $n^2$.  Combining this with the above bound on $n$ gives the claimed bound.
\end{proof}

\begin{cor}\label{cor:firstbound}
Under the conditions of Lemma~\ref{lem:finitecokeroverext}, the order of $\Br U / \Br k$ is bounded by $2^6 N(m)^2$.
\end{cor}
\begin{proof}
If we blow down the line on $X$ we find a del Pezzo surface $X'$ of degree $d+1$, such that $\Br X \cong \Br X'$. So we see in Table~\ref{table} that $\#(\Br X/\Br k) \le 64$.  Corollary~\ref{cor:isombrauergrps} gives an isomorphism $\Br X / \Br k \cong \Br_1 U / \Br k$.  Combining this with Lemma~\ref{lem:finitecokeroverext} gives the bound for $\Br U / \Br k$.
\end{proof}

Now we can prove the main theorem.

\begin{proof}[Proof of Theorem~\ref{thm}]
Let $K$ be a finite extension of $k$ such that at least one line $L$ on $\Xbar$ is defined over $K$. 
The orbit-stabiliser theorem shows that we can always take $[K:k]$ no larger than the number of lines on $\Xbar$.
Since the maximal number of lines on a del Pezzo surface is $240$, we find $[K:\mathbb Q] \leq 240m$.

By Corollary~\ref{cor:firstbound} we have a bound on $\Br U_K/\Br K$.  On the other hand, the kernel of the morphism $\Br U/\Br k \to \Br U_K / \Br K$ is contained in $\Br_1 U/\Br k$ and hence bounded by $256$ by Proposition~\ref{prop:possiblebrauers}.

Combining these two bounds, we find that
\[
\#\left(\Br U/\Br k \right)<2^{14} N(240m)^2 .\qedhere
\]
\end{proof}

\begin{remark}
There are, of course, many ways in which the constants appearing in this bound could be improved, especially if we were to separate the various different degrees.  For example, the groups $\Br U_K/\Br K$ and $\ker(\Br_1 U/\Br k \to \Br_1 U_K/\Br K)$ are far from independent.
Our interest here has been in showing the existence of a uniform bound, rather than in making that bound as small as possible.
\end{remark}

\bibliographystyle{abbrv}
\bibliography{references}

\end{document}